\documentclass[11pt,reqno]{amsart}

\usepackage{enumerate,amsmath,amssymb,amsthm} 
\usepackage{quotes}

\theoremstyle{plain}
\newtheorem{maintheorem}{Theorem}

\newtheorem{theorem}{Theorem}[section]

\newtheorem{lemma}[theorem]{Lemma}

\newtheorem{proposition}[theorem]{Proposition}

\theoremstyle{remark}

\theoremstyle{definition}
\newtheorem{definition}[theorem]{Definition}

\newcommand{\R}{\mathbb{R}}
\newcommand{\Gammatil}{\widetilde{\Gamma}}
\newcommand{\Mtil}{\widetilde{M}}
\newcommand{\htil}{\widetilde{h}}
\newcommand{\hprimetil}{\widetilde{h'}}

\DeclareMathOperator{\vol}{Vol}
\DeclareMathOperator{\area}{Area}
\DeclareMathOperator{\length}{length}

\DeclareMathOperator{\Mindeg}{Mindeg}
\DeclareMathOperator{\sys}{sys}
\DeclareMathOperator{\degree}{deg}

\DeclareMathOperator{\tens}{tens}
\DeclareMathOperator{\Ent}{Ent}
\DeclareMathOperator{\dist}{dist}

\numberwithin{equation}{section}

\title{Growth of balls in the universal cover of surfaces and graphs.}
\author{Steve Karam}
\begin{document}

\begin{abstract}
In this paper,  we prove  uniform lower bounds on the  volume growth of balls in the
universal covers of Riemannian surfaces and graphs.
More precisely, there exists a  constant $\delta>0$ such that 
if $(M,hyp)$ is a closed hyperbolic surface and $h$ another metric
on $M$ with $\area(M,h)\leq \delta \area(M,hyp)$
then for every radius $R\geq 1$ the universal cover of $(M,h)$
contains an $R$-ball with area at  least the area of an $R$-ball  in the hyperbolic plane. 
This  positively answers  a question of L. Guth for surfaces. 
We also prove an analog theorem for graphs.
\end{abstract}

\address{Steve Karam, Laboratoire de Math\'ematiques et de Physique Th\'eorique,
UFR Sciences et Technologie, Universit\'e Fran\c cois Rabelais, Parc de Grandmont,
37200 Tours, France} \email{steve.karam@lmpt.univ-tours.fr}

\maketitle

\section{Introduction}\label{secint}

Let  $(\Mtil,\htil)$ be the universal cover of a closed Riemannian manifold $(M,h)$.
We consider the function 
$$V_{(\Mtil,\htil)}(R):= \displaystyle\sup\limits_{\tilde{x}\in \Mtil} \vol B_{\htil}(\tilde{x},R).$$
The function $V(R)$ is the largest volume of any ball of radius $R$ in $(\Mtil,\htil)$.
Since it is possible to construct examples of Riemannian manifolds where 
the volume of some balls of radius $R$ in the universal cover 
is arbitrary small, it is interesting to know whether there is at least one ball of radius~$R$ 
in the universal cover with a large volume.
If the curvature of the metric~$h$ is bounded above by a negative constant 
then the Bishop-Gunther-Gromov inequality gives us an exponential lower bound on the 
volume of all balls in the universal cover $\Mtil$. So in particular we have an estimate of the function $V$.
In this paper, we are interested in finding  curvature-free exponential lower bounds for  $V$. 
We replace the local assumption, namely  a curvature bound, 
by a topological assumption  and  a condition on the volume of $(M,h)$.
What is believed is that if the topology of $M$ is complicated then the 
function $V$ is large (see \cite{G2} and \cite{G11} for more details). 
\\ \indent 
Before going further, we would like to point out that the function $V_{(\Mtil,\htil)}$ is related 
to the volume entropy of $(M,h)$. The volume entropy of $(M,h)$ is defined as 
$$\Ent(M,h)=\displaystyle\lim\limits_{R\to+\infty} \log(Vol (B_{\htil}(\tilde{x},R))).$$
Since $M$ is compact, the limit exists and does not depend on the point $\tilde{x}$ (see \cite{M}).
The volume entropy is a way of describing the asymptotic behavior 
of the volumes of  balls in the universal cover of a given Riemannian manifold.
\\

An example of a manifold with "complicated topology" is a manifold of hyperbolic type, 
i.e., a manifold which admits a hyperbolic 
Riemannian metric. Let $(M^n, hyp)$ be a closed hyperbolic manifold.
The volume of a  ball in the  hyperbolic space $\mathbb{H}^n$, i.e., the universal cover of $(M^n, hyp)$,
is independent of the center of the ball.
Thus $V_{\mathbb{H}^n}(R)$ is just the volume of any ball of radius $R$ 
in the hyperbolic $n$-space, which can be explicitly calculated. 
In particular, when $n=2$, for every $R>0$ we have  
\begin{eqnarray}
V_{\mathbb{H}^2}(R)= 2\pi(\cosh(R)-1). 
\end{eqnarray}
So there exists a constant $c$ such that 
$$V_{\mathbb{H}^2}(R)\sim ce^{R},$$
when $R$ goes to infinity.
\\ \indent
Now let $h$ be another metric on $M$ with $\vol(M,h) \leq \vol(M,hyp)$. 
Does the balls in $(\Mtil,\htil)$ also grow exponentially like in the hyperbolic case?
There exist  two fundamental theorems in this direction.
The first theorem is due to G. Besson, G. Courtois, S. Gallot \cite{BCG} and also to A. Katok \cite{K}
for the dimension $n=2$. 
The authors proved that if $M$ is a closed connected Riemannian manifold that carries 
a rank one locally symmetric metric $h_0$, then for every Riemannian metric $h$ such that 
$\vol(M,h)=\vol(M,h_0)$, the inequality $\Ent(M,h)\geq \Ent(M,h_0)$ holds.
In our language their theorem can be expressed as follows.

\begin{theorem}[see  \cite{BCG}, \cite{K}]\label{thmBessonCourtoisGallot}
Let $(M^n, hyp)$ be a closed hyperbolic manifold, and let
$h$ be another metric on $M$ with $\vol(M,h) < \vol(M,hyp)$.  
Then there is some constant $R_0$ (depending on the metric $h$)
such that for every radius $R > R_0$,  the following inequality
holds:
$$V_{(\Mtil, \htil)} (R) > V_{\mathbb{H}^n}(R).$$
\end{theorem}

\vspace{5mm}

It would be interesting to know  the value of $R_0$ in Theorem \ref{thmBessonCourtoisGallot}
since we are looking for a lower bound on the function $V_{(\Mtil, \htil)}$ for every $R\geq 0$. 
\\ 
The second fundamental theorem  can be seen  as a first step toward 
estimating $R_0$ but with a stronger hypothesis.
\\

\begin{theorem}[Guth, \cite{G11}]\label{thmGuth}
For every dimension $n$, there is a number
\linebreak
$\delta(n)>0$ such that if $(M^n, hyp)$ is a closed hyperbolic $n$-manifold and  $h$
is another metric on $M$ with  $\vol(M,h) < \delta(n)
\vol(M,hyp)$,  then  the following inequality holds
$$V_{(\Mtil, \htil)}(1) > V_{\mathbb{H}^n}(1).$$
\end{theorem}

The  method presented in \cite{G11} can be modified
to give a similar estimate for balls of radius~$R$.  
For each $R$, there is a constant
$\delta(n, R) > 0$ such that if $\vol(M,g) < \delta(n, R) \vol(M,hyp)$
then $V_{(\tilde M, \tilde g)}(R) > V_{\mathbb{H}^n}(R)$.  As
$R$ goes to infinity, the constant $\delta(n, R)$ falls off
exponentially or faster so this method become  less effective, whereas the methods in
\cite{BCG} are only effective asymptotically for very large $R$. 
This led  L. Guth to ask if we can get a  uniform estimate for $R \geq 1$. In other words, the question is: 
does there exist a positive constant $\delta(n)$ such that $\vol(M,g) < \delta(n) \vol(M,hyp)$
implies $V_{(\tilde M, \tilde g)}(R) > V_{\mathbb{H}^n}(R)$ for all $R\geq 1$?\\
\indent Here we positively answer Guth's question for the dimension $n=2$.

\begin{maintheorem}\label{maintheorem1}
There exists a positive constant $\delta$ such that if 
$(M,hyp)$ is a closed hyperbolic surface and $h$ is
another metric on $M$ with $\area(M,h)\leq \delta\area (M,hyp),$ 
then for any radius $R\geq 1$,
$$V_{(\Mtil,\htil)}(R)\geq  V_{\mathbb{H}^2}(R).$$
\end{maintheorem}

\vspace{5mm}

Our Theorem \ref{maintheorem1} will be deduced from the following  more general theorem. 

\begin{maintheorem}\label{maintheorem2}
There exists two small positive  constants $\delta$ and $c$ such that if 
$(M,hyp)$ is a closed hyperbolic surface and $h$ is
another metric on $M$ with $\area(M,h)\leq \delta\area (M,hyp),$ 
then for any radius $R\geq 0$,
$$V_{(\Mtil,\htil)}(R)\geq  V_{\mathbb{H}^2}(cR).$$
\end{maintheorem}

\vspace{5mm}

We can extend the notion of entropy from Riemannian manifolds to metric graphs.
Let $(\Gamma,h)$ be a metric graph and denote by $(\Gammatil,\htil)$ its universal cover.
Fix a point $v$ of $\Gamma$ and a lift  $\tilde{v}$ of this point in $\tilde{\Gamma}$. 
The volume entropy of $(\Gamma,d)$ is defined as 
$$ \Ent(\Gamma,h)=\displaystyle\lim\limits_{R \rightarrow \infty} \frac{\log(\length( B_{\htil}(\tilde{v},R)))}{R}.$$
Since $\Gamma$ is compact, the limit exists and does not depend on the point $\tilde{v}$ (see \cite{M}).

\begin{definition}
Let $(\Gamma,h)$ be a metric graph and denote by $(\Gammatil,\htil)$ its universal cover.
We define the function 
$$V'(R):=\sup_{\tilde{v}\in \Gammatil} \length(B_{\htil}(\tilde{v},R)),$$
where   $B_{\htil}(\tilde{v},R)$ is a ball of radius $R$ centered at the point $\tilde{v}$ of  $\Gammatil$.
\end{definition}

\vspace{5mm}

A regular graph is the analog of a Riemannian manifold carrying a locally symmetric metric.
For every positive integer $b\geq 2$, we denote by $\Gamma_b$
the connected trivalent graph of first Betti number $b$ and by $h_b$ the metric 
on $\Gamma_b$ for which all the edges have length 1. In \cite{KN} (see also \cite{L08}),
the authors  proved a theorem for graphs analog
to the G. Besson, G. Courtois  and S. Gallot theorem for manifolds.
They showed that for every integer $b\geq 2$ and every connected metric graph $(\Gamma,h)$ 
of first Betti number $b$   such that $\length(\Gamma,h)=\length(\Gamma_b,h_b)$, we have
$\Ent(\Gamma,h)\geq \Ent(\Gamma_b,h_b)$.
In our language, their theorem  can be stated as follows.

\begin{theorem}[\cite{KN},\cite{L08}] \label{thmEntropyGraphs1}
Let $(\Gamma,h)$ be a connected metric graph of first Betti number $b\geq 2$ 
Such that  $\length(\Gamma,h)< \length(\Gamma_b,h_b)$. 
Then  there exists some constant $R_0^\prime$ (depending on the metric $h$)
such that for every radius $R>R_0^\prime$ the following inequality holds 
$$V'_{(\Gammatil,\htil)}(R)\geq V'_{(\Gammatil_b,\htil_b)}(R).$$
\end{theorem}

\vspace{5mm}

In view of Theorems  \ref{thmGuth} and  \ref{thmEntropyGraphs1}, one can ask the following question: 
does there exist a universal constant $c > 0$ such that
if $\length(\Gamma,h)< c\length(\Gamma_b,h_b)$, then for all $R\geq 0$
$$V'_{(\Gammatil,\htil)}(R)\geq V'_{(\Gammatil_b,\htil_b)}(R)?$$

\vspace{5mm}

\noindent We give a partial answer to this question.

\begin{maintheorem}\label{maintheorem3}
Fix $\lambda \in (0,\frac{1}{3})$. Let $(\Gamma,h)$ be a  connected metric graph 
of first Betti number  $b\geq 2$ such that 
$$\length(\Gamma,h) \leq \lambda\length(\Gamma_b,h_b).$$ 
Then there exists a vertex $\tilde{u}$ in $\Gammatil$ such that
for any $R\geq 0$,  we have 
$$\length B_{\htil}(\tilde{u},R)\geq  (1-3\lambda)V'_{(\Gammatil_b,\htil_b)}(R).$$
In particular, we have 
$$ V'_{(\Gammatil,\htil)}(R) \geq (1-3\lambda) V'_{(\Gammatil_b,\htil_b)}(R).$$
\end{maintheorem}

\vspace{6mm}

We sketch  an outline of the main idea of the proof of  Theorem \ref{maintheorem2}. 
Fix $R\geq 0$ and denote by $g$ the genus of $M$. 
First, we  show that we can suppose that  the systole $\sys(M,h)$ of $(M,h)$ is at least $max\{2R,1/2\}$.
This lower bound on the systole and the upper bound on the area of the surface in terms of the genus permit us to show
the existence of an embedded minimal graph $\Gamma$ in $M$ which captures 
the topology of the surface (\emph{cf.}  Definition \ref{defcapture} and Definition \ref{defminimal}) 
and  satisfies the hypothesis of Theorem \ref{maintheorem3}.
Therefore, there exists a vertex $\tilde{u}$ in $\Gammatil$ such that for all 
radii $r\in(0,R)$, the length of  the ball $B_{\Gammatil}(r)$ in $\Gammatil$ centered 
at  $\tilde{u}$ and of radius $r$ is large. Since $R\leq \frac{1}{2}\sys(\Gamma,h)$, 
the length of the projection $B_{\Gamma}(r)$
of $B_{\Gammatil}(r)$ in $\Gamma$ is also large.
Let $B_M(r)$  be the ball of radius $r$ in $M$ with the same center as $B_{\Gamma}(r)$.
For all radii $r\leq R$, the boundary of $B_M(r)$ is at least as long as the graph $\Gamma \cap B_M(r)$, 
for otherwise we could construct another graph $\Gamma'$ which captures the topology of $M$ and is shorter 
than $\Gamma$. This would contradict the minimality  of $\Gamma$. 
Since the graph $\Gamma \cap B_M(r)$ contains $B_{\Gamma}(r)$, we derive that the length of $\partial{B_M(r)}$
is large. By the coarea formula, we conclude that the area of $B_R$ is also large.

\vspace{6mm}

This paper is organized as follows. 
In Section \ref{secpreliminaries}, we recall the basic material  of graphs we need in this paper.
In Section \ref{secproofbabymaintheorem3}, we  prove a special case of Theorem \ref{maintheorem3}.
In Section \ref{secproofmaintheorem3}, we prove Theorem \ref{maintheorem3} in the general case.
In Section \ref{seccapture}, we show the existence of graphs that captures the topology of closed  orientable Riemannian surfaces.
In Section \ref{secheight}, we extend the notion of the height function originally defined by Gromov for surfaces,
then we show a relation between the height and the area of balls.
In Section \ref{secexistenceregular}, we establish the existence of $\varepsilon$-regular metrics.
In Section \ref{secmini}, we define short minimal graphs on surfaces that capture the topology and we study their properties. 
At the end of this section, we show how to control their length in terms of the genus of the surface.
In Section \ref{proofofmaintheorems}, we give the proof of  the main theorems \ref{maintheorem1} and \ref{maintheorem2}.
\\

\noindent \textbf{Acknowledgment}. 
The author would like to thank his advisor, St\'ephane Sabourau, for many useful discussions and valuable comments.
He also  would like to thank Larry Guth for reading and commenting this paper.

\section{Preliminaries}\label{secpreliminaries}

By a graph $\Gamma$ we mean a finite  one-dimensional CW-complexe (multiple edges and loops are allowed).
It is also useful to see $\Gamma$ as a pair of sets $(V,E)$ where $V$ is a set of vertices 
and $E$ the set of edges, which are 2-element subsets of $V$.
Two vertices of a graph  are called \emph{adjacent} if there is an edge linking them. 
An edge and a vertex are called \emph{incident} if the vertex is an endpoint of the edge.
The \emph{degree} (also known as valence) of a vertex~$v$, denoted by $\degree(v)$, is 
the number of edges incident to it, where the loops are counted twice.
We say that a graph $\Gamma$ is \emph{$k$-regular} if the degree of any vertex is $k$. 
In particular, a  $3$-regular graph is called \emph{trivalent}. 
The minimal  degree of a graph $\Gamma$ is the minimum 
of the degrees of the vertices. It will  be denoted by $\Mindeg(\Gamma)$. 
A graph $\Gamma$ with $\Mindeg(\Gamma)\geq 3$ is called at least trivalent. 
For a graph $\Gamma$, we always denote by $E(\Gamma)$ the set of its edges and by $V(\Gamma)$ the set of its vertices.
The first Betti number  of a graph $\Gamma$ can be computed as follows:
\begin{eqnarray}
b(\Gamma)=e-v+n,
\end{eqnarray}
where $e,v$ and $n$ are respectively the number of edges, 
vertices and connected components of $\Gamma$. 
\\
\indent 
The degree sum formula states that, given a graph $\Gamma$, we have that 
\begin{eqnarray}
\displaystyle\sum\limits_{v}\degree(v)=2e,
\end{eqnarray}
where the summation is over all vertices $v$ of $\Gamma$.
\\
\indent 
For an at least trivalent connected graph $\Gamma$ with first Betti number $b$,  we have that $2e\geq 3v$  by $(2.2)$. 
Combined with $(2.1)$, we get $e\leq 3b-3$. That means that the number of edges of $\Gamma$
is bounded in terms of its first Betti number $b$.
Also it is not hard to see from $(2.1)$ and $(2.2)$ that 
every  connected graph of first Betti number $b\geq 2 $ has at least one vertex of degree at least $3$.
\\
\indent 
Let $\Gamma$ be a connected graph, $v_0$ and $v_1$ be two vertices of $\Gamma$. 
A path $P$ from $v_0$ to $v_1$ is a sequence of directed edges that links $v_0$ to $v_1$.
The vertex $v_0$ is called the start point of $P$ and $v_1$ the endpoint.
If $v_0=v_1$ then $P$ is said to be closed, otherwise $P$ is open.  
A simple path is a path with no self intersections. A simple closed path is often called a cycle.
\\

A metric graph $(\Gamma,h)$ is a graph endowed with a metric $h$ such that $(\Gamma,h)$ is a length space. 
The length of a subgraph of $\Gamma$ is its one-dimensional Hausdorff measure. 
For more details on graphs we refer the reader to \cite{DGT}.
\vspace{4mm}

Throughout this paper if $R$ is a real number then $[R]$ is the integral part of $R$.
\\

For the connected trivalent metric graph $(\Gamma_b,h_b)$ of first Betti number $b\geq 2$
where edges are of unit length, the following holds: 
\begin{itemize}
\item 
\end{itemize}

\vspace{-13mm}

\begin{eqnarray}
\hspace{-55mm}
\length(\Gamma_b,h_b)=3b-3.
\end{eqnarray}

\begin{itemize}
\item The universal cover  $\Gammatil_b$ is isometric to the trivalent infinite tree. 
In particular, $\Gammatil_b$ is independent of $b$. So for every $b'\geq 2$  we have 
$$V'_{(\Gammatil_b,\htil_b)}(R)=V'_{(\Gammatil_{b'},\htil_{b'})}(R).$$
\\
\item For every  $R\geq 0$ and every vertex $\tilde{v}$ of $(\Gammatil_b,\htil_b)$,  we have 
\begin{eqnarray}
 \length (B_{\htil_b}(\tilde{v},R))
 &=& 3\sum_{n=0}^{[R]-1}2^n +3(R-[R])2^{[R]} \nonumber \\
 &=&  3(2^{[R]}-1) + 3(R-[R])2^{[R]} \nonumber \\
 &\geq & \sinh(R\ln2).
\end{eqnarray}
Therefore, $V'_{(\Gammatil_b,\htil_b)}(R)\geq \sinh(R\ln2)$.
\\
In particular, one should notice that the volume of the ball $B_{\htil_b}(\tilde{v},R)$
is independent from the vertex $\tilde{v}$ and from the first Betti number $b$.
It only depends on $R$.
\end{itemize}

\section{Baby  theorem \ref{maintheorem3}} \label{secproofbabymaintheorem3}
In this section,  we  prove Theorem \ref{maintheorem3} with an additional bound on
the lengths of the edges of $\Gamma$ and on the minimal degree of $\Gamma$ (\emph{cf.} Section \ref{secpreliminaries}).

\begin{proposition}\label{propolip}
Let $c$ and $C'$ be two  positive constants with $c \leq C'$.
Let $(\Gamma,h)$ be a  connected, at least trivalent metric graph of first Betti number  $b\geq 2$
such that the edges of $\Gamma$ are of length  at most $c$. 
Then there exists a vertex $\tilde{u}$ in $\Gammatil$ such that for  any $R\geq 0$,  we have 
$$\length B_{\htil}(\tilde{u},(C'+c)R) \geq C' V'_{(\Gammatil_b,\htil_b)}(R).$$
In particular, we have
$$V'_{(\Gammatil,\htil)}((C'+c)R)\geq C' V'_{(\Gammatil_b,\htil_b)}(R).$$
\end{proposition}

\begin{proof}

Let $\mathcal{T}$ be a connected trivalent infinite subgraph of $\Gammatil$.
We will construct a connected trivalent infinite subgraph $\mathcal{T'}$ of $\mathcal{T}$ for which
there exists an homeomorphism $f:\Gammatil_b \to \mathcal{T'}$ that satisfies the following:\\
For every pair of vertices $x,y$ of $\Gammatil_b$, we have 
\begin{eqnarray}
C'd(x,y)\leq d(f(x),f(y)) \leq (C'+c)d(x,y).
\end{eqnarray}
For the sake of clarification, we will do this construction step by step.
\\
\\
\emph{Step 1:}
Start by fixing a vertex $v_0$ in $\mathcal{T}$.
Let $e_{1v_0}$ be one of the three edges of $\mathcal{T}$ 
incident to $v_0$ and denote by $v_{1}$ its second endpoint. 
Again let $e_{1v_1}$ be one of the other two edges of $\mathcal{T}$ 
incident to $v_{1}$ and denote by $v_{2}$ its second endpoint. 
The path $e_{1v_0}e_{1v_1}$ is  simple and open. 
We continue doing this by induction and we denote by $v_{k}$ the first vertex where 
the length of the path  $e_{1v_0}...e_{1v_k}$ is at least $C'$. 
The graph $\mathcal{T}$ contains no nontrivial cycles since it is a tree.
That means that the path $p_1=e_{1v_0}...e_{1v_k}$ is simple and open. Furthermore, 
the length of $p_1$ is between $C'$ and $C'+c$.
Now take the second edge  $e_{2v_0}$ of $\mathcal{T}$ incident to  $v_0$ and restart the  process of Step 1.
This give us another simple open path $p_2$. Again, since $\mathcal{T}$ contains no nontrivial cycles the 
intersection $p_1\cap p_2$ is the vertex $v_0$. Also restart the process with the third edge of $\mathcal{T}$ 
incident to $v_0$ to get the third path $p_3$. 
\\
\\
\emph{Step 2:}
The tree $X=p_1 \cup p_2 \cup p_3$ has three leaves. 
For each leaf $x_i$ of $X$ there are two edges of $\mathcal{T}$ incident to it 
other than the edge that is already in $X$. So by restarting the process of Step 1, 
we construct two paths  of length at least $C'$ with start point $x_i$.
By induction, we keep doing what we did before to finally get the subgraph $\mathcal{T'}$. 
In what follows each path $p_i$ of the subgraph $\mathcal{T'}$ 
will be seen as an edge of the same length of $p_i$. That means $\mathcal{T'}$
can be seen as  a connected infinite trivalent subgraph   of $\mathcal{T}$
where the length of any edge of $\mathcal{T'}$ is between $C'$ and $C'+c$.
The graphs $\Gammatil_b$ and $\mathcal{T'}$ are two  infinite trivalent trees so there exists 
an homeomorphism $f:\Gammatil_b \to \mathcal{T'}$ that sends every edge of 
$\Gammatil_b$ to an edge of $\mathcal{T'}$.
\\

Now we prove that the map $f$ satisfies $(3.1)$. Without loss of generality, 
we will prove our claim when $x$ and $y$ are the endpoints 
of the same edge $e_{xy}$ in $\Gammatil_b$, that is,  $d(x,y)=1$. 
By construction of the map $f$, the length of the image of an edge of $\Gammatil_b$ is between $C'$ and $C'+c$.
So  clearly 
$$ C'd(x,y)\leq d(f(x),f(y))\leq (C'+c)d(x,y).$$ 
Now let $\tilde{u}$ be a vertex of $\mathcal{T'}$ and denote by $w$ its inverse image in $\Gammatil_b$.  
By $(3.1)$, we have 
\begin{eqnarray}
C'\length B_{\htil_b}(w,R) &\leq & \length (f(B_{\htil_b}(w,R)) ) \nonumber \\
&\leq& \length (B_{\htil}(\tilde{u},(C'+c)R)) , \nonumber
\end{eqnarray}
Hence the proposition.
\end{proof}

\section{Proof of theorem \ref{maintheorem3}}\label{secproofmaintheorem3}
In this section, we prove Theorem \ref{maintheorem3}. As a preliminary, let us examine  
how the function $V'$ changes with scaling.
Let $(\Gamma,h)$ be a metric graph and $h'=\mu h$ with $\mu>0$ then 
\begin{itemize}
\item $\length(\Gamma,h')=\mu \length(\Gamma,h)$;
\item $V'_{(\Gammatil,\hprimetil)}(\mu R)=\mu V'_{(\Gammatil,\htil)}(R)$.
\end{itemize}

\begin{definition}\label{defignore}
Let $\Gamma$ be a connected metric graph of first Betti number at least two.
If $v$ is a vertex of $\Gamma$ of degree two then by the sentence $``\emph{ignore the vertex~$v$}$''
we mean delete the two edges $e_1$ and $e_2$ of $\Gamma$ incident to $v$ 
and replace them by
an edge of $\length\; \length(e_1)+\length(e_2)$ that  links the other two vertices of $e_1$ and $e_2$.
\end{definition}

\begin{lemma}\label{lemmatrivalent}
Let $(\Gamma,h)$ be a connected metric graph of first Betti number $b\geq 2$.
There exists a metric graph $(\Gamma',h')$ with first Betti number $b'=b$ that satisfies the following.
\begin{itemize}
\item $\Gamma'$ is at least trivalent;
\item $\length(\Gamma',h')\leq \length(\Gamma,h)$;  
\item For all $R\geq 0$,  
$$V'_{(\Gammatil',\hprimetil)}(R)\leq V'_{(\Gammatil,\htil)}(R).$$
\end{itemize}
\end{lemma}

\begin{proof}
First we remove every vertex of $\Gamma$  of degree one along with the edge incident to it 
and denote by $\Gamma_1$ the resulting connected graph. We apply the same process to $\Gamma_1$. That means 
we  remove every vertex of $\Gamma_1$  of degree one along with the edge incident to it and we denote by $\Gamma_2$
the resulting connected graph . By induction, let $\Gamma_k$ be the last connected graph where no vertex  of degree one left. 
The graph $\Gamma_k$ is of first Betti number $b$ and of length less  or equal to the length of $\Gamma$. We keep denoting
by $h$ the restriction of the metric $h$ to $\Gamma_k$.
The universal cover $\Gammatil_k$ is isometrically embedded into $\Gammatil$ so 
$$V'_{(\Gammatil,\htil)}(R) \geq  V'_{(\Gammatil_k,\htil)}(R).$$
Second, we ignore every vertex of $\Gamma_k$ of degree two (\emph{cf.} Definiton \ref{defignore}). 
The resulting graph $\Gamma'$ is 
connected of first Betti number $b$ and of the same length as $\Gamma_k$.
The universal cover $\Gammatil'$ agrees with $\Gammatil_k$ so 
$$V'_{(\Gammatil',\htil)}(R) =  V'_{(\Gammatil_k,\htil)}(R).$$
\end{proof}

In order to prove Theorem \ref{maintheorem3}, it is convenient here to reformulate it. 
Given $\lambda \in (0,\frac{1}{3})$, let $c$ and $C'$ be two positive constants 
such that $c\leq C'$ and $\lambda=\frac{c}{3(C'+c)}$. 
So a reformulated version of Theorem \ref{maintheorem3} is the following.

\begin{theorem}
Let $(\Gamma,h)$ be a  connected metric graph of first Betti number  $b\geq 2$.
Let $C'$ and $c$ be two  positive constants with $c\leq C'$. Suppose that 
$$\length(\Gamma,h)\leq \frac{c}{3(C'+c)}\length(\Gamma_b,h_b).$$ 
Then there exists a vertex $\tilde{u}$ in $\Gammatil$ such that for any $R\geq 0$,  we have
$$\length B_{\htil}(\tilde{u},R)\geq \frac{C'}{C'+c} V'_{(\Gammatil_b,\htil_b)}(R).$$
In particular, we have  
$$ 
V'_{(\Gammatil,\htil)}(R) \geq \frac{C'}{C'+c} V'_{(\Gammatil_b,\htil_b)}(R).
$$
\\
\end{theorem}

\begin{proof}
By  scaling, we will prove the following. Suppose that 
$$\length(\Gamma,h)\leq \frac{c}{3}\length(\Gamma_b,h_b)=c(b-1).$$ 
Then there exists a vertex $\tilde{u}$ in $\Gammatil$ such that for any $R\geq 0$,  we have
$$\length B_{\htil}(\tilde{u},(C'+c)R)\geq C' V'_{(\Gammatil_b,\htil_b)}(R).$$
In particular, we have 
$$ 
V'_{(\Gammatil,\htil)}((C'+c)R) \geq C' V'_{(\Gammatil_b,\htil_b)}(R).
$$
\\
\indent  First notice that by Lemma \ref{lemmatrivalent}, we can suppose that $\Gamma$ is at least trivalent.
We proceed by induction on the first Betti number of $\Gamma$.
For $b=2$, we have 
$$\displaystyle\max\limits_{e\in E} \length(e)<\length(\Gamma,h)\leq  c(2-1)=c.$$  
By Proposition \ref{propolip}, the result follows in this case. 
\\
Suppose the result holds for $b=n$ and let us show that it also for $b=n+1$.
Let $(\Gamma,h)$ be a connected metric graph of first Betti number $b=n+1$.
If $\Gamma$ contains no edge of length greater than $c$ then the result follows 
from Proposition \ref{propolip}. Thus we suppose the opposite here and  
remove an edge $w$ of $\Gamma$ of length greater than $c$. There are two cases to consider.
\\
\\
\emph{Case 1}: The edge $w$ is non-separating in $\Gamma$.
In this case, the resulting graph $\Gamma'$ is connected and of first Betti number $b'=n$. 
Furthermore, we have  
$$\length(\Gamma')\leq \length(\Gamma)-c \leq c(b'-1).$$
The universal cover $\Gammatil'$ is isometrically embedded into $\Gammatil$.
So for every vertex $\tilde{v}$ in $\Gammatil'$ and every $R>0$, we have
$$\length (B_{(\Gammatil,\htil)}(\tilde{v},R)) \geq \length (B_{(\Gammatil',\htil)}(\tilde{v},R)).$$
In particular, we have
$$V'_{(\Gammatil,\htil)}(R) \geq  V'_{(\Gammatil',\htil)}(R).$$
On the other hand,  by the hypothesis of the induction, we know that there exists a vertex $\tilde{u}$ in $\Gammatil'$
such that 
$$\length( B_{(\Gammatil',\htil)}(\tilde{u},R))\geq V'_{(\Gammatil_n,\htil_n)}(R)=V'_{(\Gammatil_{n+1},\htil_{n+1})}(R).$$
In particular, we have 
$$V'_{(\Gammatil',\htil)}(R)\geq V'_{(\Gammatil_{n+1},\htil_{n+1})}(R).$$
This finishes the proof in this case.
\\
\\
\emph{Case 2}: The edge $w$ is separating in $\Gamma$.
Thus, it splits the graph $\Gamma$ into two connected graphs $\Gamma'$ and $\Gamma''$ of first Betti number  
 $b'$ and $b''$.
We claim that $\length(\Gamma')\leq c(b'-1)$ or $\length(\Gamma'')\leq c(b''-1)$. Indeed, suppose the opposite 
then 
$$\length(\Gamma')+\length(\Gamma'')>c(b-2).$$
On the other hand we have 
$$\length(\Gamma')+\length(\Gamma'')+c<\length(\Gamma)\leq c(b-1).$$
Hence a contradiction. So the claim is proved.
\\
Without loss of generality, suppose that $\Gamma'$ satisfies $\length(\Gamma')\leq c(b'-1)$. 
Clearly $b'\geq 2$, otherwise the length of $\Gamma'$ would vanish. By induction, we now 
 there exists a vertex $\tilde{u}$ in $\Gammatil'$ such that 
$$\length( B_{(\Gammatil',\htil)}(\tilde{u},R))\geq V'_{(\Gammatil_{b'},\htil_{b'})}(R)=V'_{(\Gammatil_{b},\htil_{b})}(R).$$
In particular, we have  
$$V'_{(\Gammatil',\htil)}(R)\geq V'_{(\Gammatil_{b},\htil_{b})}(R).$$
Recall that the universal cover $\Gammatil'$ is isometrically embedded into $\Gammatil$. So 
for every vertex $\tilde{v}$ in $\Gammatil'$ and every $R>0$, we have
$$\length (B_{(\Gammatil,\htil)}(\tilde{v},R)) \geq \length (B_{(\Gammatil',\htil)}(\tilde{v},R)).$$
In particular, we have
$$V'_{(\Gammatil,\htil)}(R) \geq  V'_{(\Gammatil',\htil)}(R).$$
This finishes the proof in this case too.

\end{proof}

\section{Capturing the topology of surfaces}\label{seccapture}
In this section, we show that on every closed orientable Riemannian
surface $M$ there exist an  embedded graph that captures its topology.

\begin{definition}\label{defcapture}
Let $(M,h)$ be a  closed Riemannian surface of genus $g$.
The image in $M$ of an abstract graph by an  embedding will be refered to as a graph in~$M$.
The metric $h$ on $M$ naturally induces a metric  on a graph $\Gamma$ in~$M$.
Despite the risk of confusion,  we will also denote by $h$ such a metric on $\Gamma$.
\\
\indent We  say that a graph $\Gamma$ in $M$ \emph{captures the topology} of $M$ if the map 
induced by the inclusion $i_\ast :H_1(\Gamma,\R) \rightarrow H_1(M,\R)$ is an epimorphism.
\\
\end{definition}

\begin{lemma}\label{bettigenus}
Let $(M,h)$ be a closed orientable Riemannian surface.
Let  $\Gamma$ be a connected  graph in $M$ that captures its topology  and 
denote by $i:\Gamma \to M$ the inclusion map.
Then there exists a connected subgraph $\Gamma'$ of $\Gamma$ such that 
the map $i_\ast$ restricted to $\Gamma'$ is an isomorphism.
In particular the first Betti number of $\Gamma'$ is $2g$.
\end{lemma}

\begin{proof}
Let $\Gamma'$ be a connected subgraph of $\Gamma$ with minimal number of edges such 
that the restriction of $i$ to $\Gamma'$ still induces an epimorphism in real homology.
Let $\alpha$ be a cycle of $\Gamma'$ representing a nontrivial element of the kernel of $i_\ast$.
Remove an edge $e$ from $\alpha$. 
The resulting graph $\Gamma''$ has fewer edges than $\Gamma'$.
Let $\beta$ be a cycle of $\Gamma'$.
If $e$ does not lie in $\beta$ then the cycle $\gamma=\beta$ lies in $\Gamma''$. 
Otherwise, adding a suitable real multiple of $\alpha$ to $\beta$ yields a new cycle $\gamma$ lying in $\Gamma''$. In
both cases, the cycle $\gamma$ of $\Gamma''$ is sent to the same homology class as $\beta$ by $i_\ast $.
Thus, the restriction of~$i$ to~$\Gamma''$ still induces an epimorphism in the real 
homology, which is absurd by definition of~$\Gamma'$.
\end{proof}

\vspace{5mm}

In what follows a graph $\Gamma$ in a Riemannian manifold $(M,h)$ is automatically 
equipped with the metric $h$ induced by the  metric of $M$. So the length of $\Gamma$ 
is its one-dimensional Hausdorff measure associated to the metric $h$.

\begin{definition}
Let $(M,h)$ be a closed orientable Riemannian surface. 
We define  
$$L(M,h):= \displaystyle\inf\limits_{\Gamma}\length(\Gamma),$$
where the infimum is taken over all  graphs $\Gamma$ in $M$ that capture its  topology. 
\\
\end{definition}

\begin{lemma}
Let $(M,h)$ be a closed orientable Riemannian surface of genus~$g$. 
Then there exists a graph   $\Gamma$ in $M$ that captures its topology with
$$\length(\Gamma)=L(M,h).$$
\end{lemma}

\begin{proof}
By Lemma \ref{bettigenus},   we only need to consider the set of graphs 
in $M$ that captures its topology with first Betti number $2g$ and such that $i_\ast$ is an isomorphism.
Furthermore,  we only need to consider graphs that are at least trivalent.
Indeed, delete every vertex of $\Gamma$  of degree one along with the edge incident to it. 
Denote by $\Gamma_1$ the resulting connected graph and apply to $\Gamma_1$ the same process. That means 
we  delete every vertex of $\Gamma_1$  of degree one along with the edge incident to it and we denote by $\Gamma_2$
the resulting connected graph . By induction, let $\Gamma_k$ be the last connected graph with no vertex of degree one. 
We then ignore all vertices of $\Gamma_k$ of degree two (\emph{cf.} Definition \ref{defignore}). Replacing every edge of $\Gamma_k$ by a minimal representative
of its fixed-endpoint homotopy class gives rise to a geodesic graph $\Gamma'$. By construction the connected geodesic graph 
$\Gamma'$ is at least trivalent and of first Betti number $2g$. Thus, its number of edges is bounded in terms of $g$, 
\emph{cf.} Section \ref{secpreliminaries}. Now the space of connected geodesic graphs of $M$ capturing its topology
with bounded length and a bounded number of edges is compact. The result follows.
\end{proof}

\begin{definition}\label{defminimal}
Let $(M,h)$ be a closed orientable Riemannian surface. 
If $\Gamma$ is a graph that captures the topology of  $M$ with $\length(\Gamma)=L(M,h)$, 
then $\Gamma$ is called a minimal graph in $M$.
\end{definition}

\section{Height function and area of balls.}\label{secheight}

In this section, we first recall the definition of the \emph{height} function 
on surfaces defined by Gromov in \cite{G1} along with its relation to the area of balls. 
Then we extend this notion to make it suit our problem.
\\
\\
Let $M$ be a closed Riemannian manifold. 
The systole  at a point $x$ in $M$, denoted by $\sys(M,x)$, is the length of the shortest non-contractible loop based at $x$. 
The systole of $M$, denoted by  $\sys(M)$, is  the length of the shortest non-contractible loop in $M$.

\begin{definition}\label{definitiontension}
Let $(M,h)$ be a closed Riemannian surface and $\gamma$ be a  non-contractible loop in $M$.
We define the tension of $\gamma$ as follows.
$$\tens(\gamma)=\length(\gamma) - \displaystyle\inf\limits_{\beta\sim \gamma}(\length(\beta)),$$
where the infimum is taken  over all closed curves $\beta$ freely homotopic to $\gamma$.
\\
\indent We also define the height  function $H'$ on $M$ as follows
$$H'(x)=\displaystyle\inf\limits_\gamma(\tens(\gamma)),$$
where the infimum is taken over all non-contractible closed curves $\gamma$ passing through $x$.
\end{definition}

\begin{proposition}[Gromov, \cite{G1} Proposition 5.1.B]\label{areaballsgromov}
Let $(M,h)$ be a complete Riemannian surface and $x\in M$.
Then 
$$\area B(x,R)\geq \frac{1}{2}(2R-H'(x))^2,$$
for every $R$ in the interval  $[\frac{1}{2}H'(x),\frac{1}{2}\sys(M,x)]$.
\\
\end{proposition}

\vspace{5mm}

\begin{definition}
Let $(M,h)$ be a closed orientable Riemannian surface. 
For $x \in M$, we define
$$L(M,x) := \displaystyle\inf\limits_{\Gamma_x}\length(\Gamma_x),$$
where the infimum is taken over  all  graphs $\Gamma_x$ in $M$ that capture its topology and  pass through $x$.
\\  
We also define the function $H''$ on $M$ as follows.
$$H''(x):=L(M,x)-L(M,h).$$
Finally we define the function $H$ on $M$ as
$$H(x):=min(H'(x),H''(x)),$$
where $H'$ is defined in Definition \ref{definitiontension}.
\end{definition}

\begin{definition} \label{defcell}
If $B$ is a ball in some closed Riemannian surface $M$ 
with some contractible boundary components, we fill in every such component 
of $\partial B$ by an open $2$-cell in $M$ and denote by $B^+$ the union of $B$ with these cells.
\end{definition}

\begin{proposition}\label{smallballsvolume}
Let $(M,h)$ be a closed Riemannian surface of genus $g\geq 1$
and $x\in M$ with $H(x)<\frac{1}{2}\sys(M,x)$. 
Then the area of the ball $B(x,R)$  satisfies the inequality 
$$\area\;B(x,R)\geq \frac{1}{2}(R-H(x))^2,$$
for every $R$ in the interval
$]H(x),\frac{1}{2}\sys(M,x)[.$
\\
\end{proposition}

\begin{proof}
We  suppose that $H(x)=H''(x)$ here, since the other case follows from
Proposition \ref{areaballsgromov}.
Let $r\in ]H(x),\frac{1}{2}\sys(M,x)[.$
Notice that since $r<\frac{1}{2}\sys(M,x)$ the ball $B=B(x,r)$ is contractible in $M$, 
and so the set $B^+=B^+(x,r)$ is a topological disk.
Let $\varepsilon$ be a fixed small positive constant such that $H''(x)+\varepsilon<r$. 
Fix $\varepsilon' \in (0,\varepsilon)$. Let $\Gamma_x$  be a graph   
in $M$ that captures its topology and  passes through $x$ of length at most 
$L(M,x)+\varepsilon'$. \linebreak
Without loss of generality,  we claim that we can always suppose that 
$\Gamma_x \cap B^+(x,r)$ 
is a tree such that $x$ is the only possible vertex of degree one.
Indeed, we  delete an edge from each loop of $\Gamma_x\cap B^+(x,r)$. This defines a new graph $\Gamma'$.
Then we delete every vertex of $\Gamma'$  of degree one other than the vertex $x$ 
along with the edge incident to it and we 
denote by $\Gamma_1$ the resulting connected graph. Restart the process. That means 
we  delete every vertex of $\Gamma_1$  of degree one other than the vertex $x$ 
along with the edge incident to it and we denote by $\Gamma_2$
the resulting connected graph. By induction, let $\Gamma_k$ be the 
last connected subgraph where the only possible vertex of degree one is $x$.
Clearly $\Gamma_k$ passes through $x$, captures the topology of $M$ and is of length  at
most $L(M,x)+\varepsilon'$. So the claim is proved.
\\

Now we claim that  either $x$ is of degree at least two or there is at least a vertex 
of $\Gamma_x \cap B^+$ of degree at least three.
Indeed, suppose that $x$ is of degree one and all the other vertices of $\Gamma_x \cap B^+$ are of degree two.
Then  $\Gamma_x\cap B^+$ is just a piecewise 
curve that passes through
$x$ and hits  $\partial B^+$  at one point, so  
its length is greater or equal to $r$.
Thus 
$$\length(\Gamma_x)\geq L(M,h)+r.$$
In particular, we have 
$$L(M,h)+ r \leq L(M,x)+\varepsilon' \leq L(M,x)+\varepsilon.$$
That means 
$$r \leq H''(x)+\varepsilon,$$
which is  a contradiction.
\\
\\
In both cases above, the graph $\Gamma_x$ hits 
the boundary of $B^+$ in at least two points.
Let $C$ be a minimal arc of $\partial B^+$ that connects 
the  points of $\Gamma_x\cap \partial B^+$. 
Consider the graph  $\Gamma'$  defined as $(\Gamma_x\setminus (\Gamma_x\cap B^+))\cup C$. 
It is clear that $\Gamma'$ is a connected  graph in $M$ that captures its topology, since $B^+$ is contractible in $M$.
Thus
$$\length(\Gamma')\geq L(M,h).$$ 
On the other hand, the length of $\Gamma_x\cap B^+$ is at least $r$.
This means that 
$$\length(\Gamma_x)\geq \length(\Gamma')+r-\length(C).$$
So
$$L(M,x) +\varepsilon' \geq L(M,h) +r-\length(C).$$
We conclude that for every small positive constant $\varepsilon'$, we have
$$H''(x)\geq r-\length(C)-\varepsilon'.$$
Since the length of $\partial B^+$ is at least  the length of the arc $C$, we have
$$\length (\partial B^+)\geq r- H''(x).$$
By the coarea formula,
\begin{eqnarray}
\area B(x,R)&\geq & \int_0^R \length(\partial B(x,r))dr \nonumber \\
&\geq & \int_{H''(x)}^R \length(\partial B^+(x,r))dr \nonumber \\
&= & \frac{1}{2}(R-H''(x))^2.\nonumber
\end{eqnarray}
\end{proof}

\section{Existence of $\varepsilon$-regular metrics.}\label{secexistenceregular}
In this section, we define $\varepsilon$-regular metrics and  prove their existence.
The existence of $\varepsilon$-regular metrics will play a crucial role in 
controlling  the length of  minimal graphs on surfaces.

\begin{definition}
Let $(M ,h)$ be a closed Riemannian surface. 
The metric $h$ is called $\varepsilon$-regular if for all the points $x$ in $M$, $H(x)\leq \varepsilon.$
\\
\end{definition}

\begin{lemma}\label{lemmeexistregular}

Let $(M_0,h_0)$ be a closed Riemannian surface. 
Then for every  $\varepsilon>0 $, there exists a Riemannian   
metric $\bar{h}$ on $M_0$ conformal to $h_0$  such that
\begin{enumerate}
\item  $\area(M_0,\bar{h}) \leq \area(M_0,h_0);$
\item  $\bar{h}$ is $\varepsilon$-regular; 
\item  $L(M_0,\bar{h})=L(M_0,h_0)$;
\item  $\sys(M_0,\bar{h})=\sys(M_0,h_0).$
\end{enumerate}
\end{lemma}

\begin{proof}
Take a point $x_0$ in $M_0$ where $H(x_0)=H_{h_0}(x_0)> \varepsilon$  
and denote by $M_1$  the space $M_0/B^+$ obtained by collapsing  $B^+=B^+(x_0,\varepsilon$) to $x_0$.
Let $p_0: M_0\to M_1$ be the (non-expanding) canonical projection 
and  $h_1$ be the metric induced by $h$ on $M_1$. 
The Riemannian surface $(M_1,h_1)$ clearly satisfies $(1)$. 
If $h_1$ is not $\varepsilon$-regular, we apply the same process. 
By induction we construct a sequence of :
\begin{itemize}
\item balls $B^+_i=B^+(x_i,\varepsilon)$ in $M_i$, where  $x_i$ is a point with $H_{h_i}(x_i)> \varepsilon$.
\item Riemannian surfaces $(M_i,h_i)$ where 
$M_i=M_{i-1}/B_{i-1}$ and $h_i$ is the metric induced by $h_{i-1}$ on $M_i$.
\item non-expanding canonical projections $p_i:M_i\to M_{i+1}$.
\end{itemize}
This process stops when we get an $\varepsilon$-regular metric.
\\

Now,  we argue exactly as  \cite[Lemma 4.2]{YS} to prove that this process stops after finitely many steps.
Let $B_1^i,\dots,B_{N_i}^i$ be a maximal system of disjoint
balls of radius $r/3$ in $M_i$.
Since $p_{i-1}$ is non-expanding, the preimage $p_{i-1}^{-1}(B_k^i)$
of $B_k^i$ contains a ball of radius $r/3$ in $M_{i-1}$.
Furthermore, the preimage $p_{i-1}^{-1}(x_i)$ of $x_i$
contains a ball $B_{i-1}$ of radius $r$ in $M_{i-1}$.
Thus, two balls of radius $r/3$ lie in the preimage of $x_i$
under $p_{i-1}$.
It is then possible to construct a system of $N_i+1$ disjoint
disks
of radius $r/3$ in $M_{i-1}$.
Thus, $N_{i-1} \geq N_i +1$ where $N_i$ is the maximal number
of disjoint balls of radius $r/3$ in $M_i$.
Therefore, the process stops after $N$ steps with $N \leq
N_0$. Denote by $h_N$ the metric where this process stops. Clearly $h_N$ satisfies (1) and~(2).
To see that $h_N$ satisfies $(3)$ and $(4)$, 
let $\Gamma$  be a minimal graph in $M_0$ and $\alpha$ be a systolic loop in $M$. 
For every  point $x$ in  the $\varepsilon$-neighborhood $N_{\Gamma}$ of $\Gamma$,
we have $H(x)\leq \varepsilon$. 
Indeed, let $c$ be a minimizing curve from $\Gamma$ to $x$. 
The graph $\Gamma\cup c$ captures the topology of $M_0$ and  passes through~$x$. 
So $H''(x)\leq \length(\Gamma\cup c)-L(M,h)\leq \varepsilon$. 
That means that the balls we collapsed  through the whole process do not intersect $\Gamma$.
Therefore, the metric $h_N$ satisfies $(3)$. A similar argument holds for $\alpha$. 
So the metric $h_N$ also satisfies  $(4)$.

\end{proof}

\section{Construction of short minimal graphs on surfaces}\label{secmini}
In this section, we combine  Lemma \ref{lemmeexistregular} and the construction of \cite[p. 46]{BPS} to 
construct a minimal graph with controlled length  on a given Riemannian surface.

\begin{proposition}\label{lengthofminimalgraph}
Let $(M,h)$ be a closed orientable Riemannian surface of genus $g \geq 2$.
Suppose that
\begin{itemize}
\item $\area(M,h) \leq \frac{1}{2^{12}}(2g-1)$;
\item $\sys(M,h) \geq \frac{1}{2}$.
\end{itemize}
Then 
$$L(M,h)\leq \frac{1}{2}(2g-1).$$
\end{proposition}

\begin{proof}

Fix $r_0=\frac{1}{2^5}$.
By Lemma \ref{lemmeexistregular} (choose $\varepsilon$ small enough) and Proposition \ref{smallballsvolume}, 
there exists a conformal Riemannian metric $\bar{h}$ on $M$
that satisfies
\begin{enumerate}
\item The area of every disk of $(M,\bar{h})$ of radius $r_0$ is at least $\frac{1}{4}{r_0}^2;$ 
\item $\area(M,\bar{h})\leq \area(M,h);$
\item $L(M,\bar{h})=L(M,h)$;
\item $\sys(M,\bar{h})=\sys(M,h)$
\item $\bar{h}$	is $\varepsilon$-regular.
\end{enumerate}
So it is sufficient to prove that 
$$L(M,\bar{h})\leq\frac{1}{2}(2g-1).$$
Let $\{B_i\}_{i \in I}$ be a maximal system of disjoint balls of radius $r_{0}$ in $(M,\bar{h})$.
Since the area of each ball $B_i$ is at least $\frac{1}{4}r_{0}^2$, then 
$$\frac{1}{4}|I|r_{0}^2 \leq \area(M,\bar{h}),$$
that is,
\begin{eqnarray}
|I|  \leq  2^{12}\area(M,\bar{h}).
\end{eqnarray}

\noindent As this system is maximal, the balls $2B_{i}$ of radius $2r_{0}$ 
with the same centers $p_{i}$ as $B_{i}$ cover $M$. 
\\
\indent 
Let $\varepsilon$ be a small positive constant that satisfies 
$$4r_{0}+2\varepsilon < \frac{1}{4} \leq \frac{\sys(M,\bar{h})}{2},$$
and denote by $2B_{i}+\varepsilon$  the balls centered at $p_{i}$ with radius $2r_{0}+\varepsilon$.
We construct an abstract graph $\Gamma$ as follows. Let   $\{w_{i}\}_{i \in I}$ be a set of vertices 
corresponding to $\{p_{i}\}_{i \in I}$.  
Two vertices  $w_{i}$ and $w_{i'}$ of $\Gamma$ are linked by an edge if and only if 
the balls $2B_{i}+\varepsilon$ and $2B_{i'}+\varepsilon$ intersect each other.
Define a metric on $\Gamma$  such that the length of each edge is $\frac{1}{4}$ and let 
$\varphi:\Gamma \to M$ be the map that sends each edge of  $\Gamma$ with endpoints $w_i$ and $w_{i'}$ to a minimizing 
geodesic joining $p_i$ and~$p_{i'}$. 
Since  
$\dist(p_i,p_{i'})\leq 4r_{0}+2\varepsilon < \frac{1}{4}$, the map $\varphi$ is distance nonincreasing.
\\
\\
\textbf{Claim}. The  map $  {\varphi}_{\star}: \pi_{1}(\Gamma) \to \pi_{1}(M)$
induced by $\varphi$ between the fundamental groups is an epimorphism.
In particular, it induces an epimorphism in real homology.
\\
\\
We argue exactly as \cite[Lemma 2.10]{BPS}.
Consider  a geodesic loop $\sigma$ of $M$.
Divide the loop $\sigma$ into segments $\sigma_{1},\ldots, \sigma_{n}$ of length at most~$\varepsilon$.
Denote by $x_{k}$ and~$x_{k+1}$ the endpoints of $\sigma_{k}$ with the convention $x_{n+1}=x_{1}$.
Recall that the balls $2B_{i}$ cover the surface $M$. So every point $x_{k}$ is at distance 
at most $2r_{0}$ from a point $v_{k}$ among the centers $p_{i}$.
Let $\beta_k$ be the loop
$$
\sigma_k \cup C_{x_{k+1}v_{k+1}} \cup C_{v_{k+1},v_k} \cup C_{v_k,x_k},
$$
where $C_{ab}$ denotes a minimizing geodesic  joining $a$ to $b$. 
We have that 
$$
\length(\beta_k) \leq 2 \, (4 r_{0} + \varepsilon) < \sys(M,\bar{h}).
$$ 
That means that the loops $\beta_k$ are contractible.
We conclude that the loop $\sigma$ is homotopic to a piecewise geodesic loop $\sigma'= (v_1, \dots, v_{n})$.
\\
The distance between the centers $v_{k}=p_{i_{k}}$ and $v_{k+1}=p_{i_{k+1}}$ 
is less than or equal to $4r_{0}+\varepsilon$.
So  the vertices $w_{i_{k}}$ and $w_{i_{k+1}}$ of $\Gamma$ corresponding 
to the vertices $p_{i_{k}}$ and $p_{i_{k+1}}$ are connected by an edge.
The union of these edges forms a loop $(w_{i_{1}}, \ldots, w_{i_{n}})$ 
in $\Gamma$ whose image by the map $\varphi$ is $\sigma'$. 
Since $\sigma'$  is homotopic to  $\sigma$,
the claim is proved.
\\
\\
Now we consider a connected subgraph $\Gamma'$ of $\Gamma$ with a
minimal number of edges such that the restriction of $\varphi$ 
to $\Gamma'$ still induces an epimorphism in real homology.
\\

We claim that the epimorphism $\varphi_\ast:H_{1}(\Gamma';\R) \to H_{1}(M;\R)$  is an isomorphism. 
Indeed, if $\varphi_{\ast}$ is not an isomorphism then  
arguing as in Proposition \ref{bettigenus} we can remove at least one edge 
of $\Gamma'$ such that $\varphi_{\ast}$ is still an epimorphism, 
which is impossible by the definition of $\Gamma'$.
\\
We denote by $v,e,b$ and $b'$ respectively the number of vertices of 
$\Gamma$, the number of edges of $\Gamma$, the first Betti number of $\Gamma$ and the first Betti number of $\Gamma'$.
At least $b-b'$ edges were removed from $\Gamma$ to obtain $\Gamma'$.
As $b'=2g$, we derive
\begin{eqnarray}
\length(\Gamma') & \leq & \length(\Gamma) - (b-b') \, \frac{1}{4} \nonumber \\
 & \leq & (e-b+2g) \, \frac{1}{4} \nonumber \\
 & \leq & (v-1+2g) \, \frac{1}{4}. 
\end{eqnarray}
Recall that  
$\area(M,\bar{h})\leq \frac{1}{2^{12}}(2g-1)$. 
So 
$$v=|I|\leq 2g-1.$$
Combining this  with $(6.2)$, we get
$$\length(\Gamma')\leq \frac{1}{2}(2g-1).$$
Since $\varphi$ is distance non-increasing then
$$\length(\varphi(\Gamma'))\leq \length(\Gamma').$$

The image by   $\varphi$ of two edges of $\Gamma'$ may intersect. 
If it is the case then the intersection point should be considered as a vertex of the 
graph $\varphi(\Gamma')$. Thus the set of vertices of $\varphi(\Gamma')$ may be bigger 
than the set of vertices of $\Gamma'$.
\\

Finally let  $j$ be the inclusion map $j:\varphi(\Gamma') \hookrightarrow M$. 
Clearly the map $j_\ast:H_{1}(\varphi(\Gamma');\R) \to H_{1}(M;\R)$ is an epimorphism. 
So $\varphi(\Gamma')$ is a graph in $M$
that captures its topology. 
Thus 
$$L(M,\bar{h})\leq \length(\varphi(\Gamma')) \leq \frac{1}{2}(2g-1).$$

\end{proof}

\section{Proofs of Theorem \ref{maintheorem1} and Theorem \ref{maintheorem2}.}\label{proofofmaintheorems}

In this section, we prove  Theorem \ref{maintheorem1} and Theorem \ref{maintheorem2}. But before doing that we examine how 
the function $V$ changes with scaling.
Let $(M^n,h)$ be a closed $n$-dimensional Riemannian manifold and $h'=\lambda^2h$ with $\lambda>0$ then 
\begin{itemize}
\item $\vol(M,h')=\lambda^n \vol(M,hyp)$;
\item $V_{(\Mtil,\hprimetil)}(\lambda R)=\lambda^n V_{(\Mtil,\htil)}(R)$.
\end{itemize}

\vspace{5mm}

The expression (1.1) of $V_{\mathbb{H}^2}$ immediately leads to the following lemma.

\begin{lemma}\label{lemmascal}
Let $a$ be a positive constant. There exists a constant $c=c(a)$ such that 
for all $R\geq 0$,
$$a V_{\mathbb{H}^2}(R)\geq  V_{\mathbb{H}^2}(Rc).$$
\end{lemma}

\vspace{4mm}

In light of Lemma \ref{lemmascal},  the proof of Theorem \ref{maintheorem2} amounts to proving  the following result.

\begin{theorem}\label{mainthm}
Let $(M,hyp)$ be a closed hyperbolic surface of genus $g$ and $h$ 
be another Riemannian metric on $M$ with 
$$\area(M,h)\leq \frac{1}{2^{13} \pi}\area (M,hyp).$$ 
Then, for any radius $R\geq 0$,
$$V_{(\Mtil,\htil)}(R)\geq \frac{1}{4\pi\ln2}V_{\mathbb{H}^2}(R\ln2).$$
In particular, there exists a constant $c$ such that 
$$V_{(\Mtil,\htil)}(R)\gtrsim c\,2^R,$$
when $R$ tends to infinity.
\\
\end{theorem}

\begin{proof}
Let $R>0$. First, we consider the special case when $M$ is oriented and   
$$\sys(M,h)\geq  max\{2R,1/2\}.$$ 
In this case,
$$V_{(M,h)}(R)=V_{(\Mtil,\htil)}(R).$$ 
Let $\Gamma$ be a minimal  graph which captures the topology of  $(M,h)$ (\emph{cf}. Definition \ref{defminimal}). 
Denote by $b=2g$ the first  Betti number of $\Gamma$.
We have 
$$\area(M,h)\leq \frac{1}{2^{13}\pi}\area(M,hyp)\leq \frac{1}{2^{12}}(2g-1).$$
So by Proposition \ref{lengthofminimalgraph} and the relation $(2.3)$, we have  
\begin{eqnarray}
\length(\Gamma)\leq \frac{1}{2}(b-1)=\frac{1}{6}\length(\Gamma_b,h_b).
\end{eqnarray}
Let $v$ be any vertex of $\Gamma$. Denote by $B(v,R)$ 
the ball in $(M,h)$ centered at~$v$  with radius $R$.
We claim that  for all  $r \in  (0,R)$
\begin{eqnarray}
\length(\partial B^+(v,r)) \geq  
\length(\Gamma\cap B^+(v,r)),
\end{eqnarray}
where $B^+(v,r)$ is defined in Definition \ref{defcell}.
\\
We argue as in Proposition \ref{smallballsvolume}. 
Suppose the opposite and 
replace $\Gamma\cap B^+(v,r)$
by a minimal arc of $\partial B^+(v,r)$ that links the points
of $\Gamma\cap \partial B^+(v,r)$. Since $B^+(v,r)$ is contractible,
the new graph  captures the topology of $M$ and is shorter 
than $\Gamma$ which contradicts  the definition of $\Gamma$.
\\
\\
Let $B_{(\Gamma,h)}(v,r)$ be the ball centered at $v$  
of radius $R$ in the metric graph $(\Gamma,h)$.
Since the ball $B_{(\Gamma,h)}(v,r)$ 
is contained in $\Gamma\cap B^+(v,r)$, we have 
\begin{eqnarray}
\length(\Gamma\cap B^+(v,r))
\geq 
\length (B_{(\Gamma,h)}(v,r)).
\end{eqnarray}
Let $\tilde{v}$ be a lift of $v$ in $\Gammatil$.
Since $\sys(M,h)\leq \sys(\Gamma,h)$,
we have for $r \leq~\frac{1}{2}\sys(M,h)$ 
\begin{eqnarray}
\length (B_{(\Gamma,h)}(v,r))=\length (B_{(\Gammatil,\htil)}(\tilde{v},r)).
\end{eqnarray}
By Theorem \ref{maintheorem3} (take $\lambda= \frac{1}{6}$) and the bound $(9.1)$, 
there exists a vertex $\tilde{u}$ in $\Gammatil$ such that 
$$\length (B_{(\Gammatil,\htil)}(\tilde{u},r))
\geq
\frac{1}{2}V'_{(\Gammatil_{2g},\htil_{2g})}(r).$$
Denote by $u$ the image of $\tilde{u}$ by the covering map. 
By $(9.2)$, $(9.3)$, $(9.4)$ and $(2.4)$, we obtain
\begin{eqnarray}
\length(\partial B^+(u,r)) &\geq & 
\frac{1}{2} V'_{(\Gammatil_{2g},\htil_{2g})}(r) \nonumber \\
&\geq& \frac{1}{2}\sinh(r\ln 2).\nonumber
\end{eqnarray}
By the coarea formula,
\begin{eqnarray}
\area(B(u,R)) &\geq& \frac{1}{2}\int_0^R \sinh(r\ln 2) dr \nonumber \\
&=& \frac{1}{2\ln 2 }(\cosh(R\ln2)-1). \nonumber \\
&=& \frac{1}{4\pi \ln2}V_{\mathbb{H}^2}(R\ln2). \nonumber 
\end{eqnarray}

Next, we consider the general case  with no restriction on the systole and the orientability of $M$.
Since $M$ admits a hyperbolic metric, 
the fundamental group of $M$ is residually finite (see \cite{Ma}). 
Therefore, we can choose a finite cover  
$(\bar{M},\bar{h})$ such that $\bar{M}$ is orientable 
and $$\sys(\bar{M},\bar{h}) \geq  max\{2R,1/2\}.$$ 
Let $\bar{hyp}$  be the pullback of the hyperbolic  metric on $M$ to $\bar{M}$.  
\\
Now, if the covering $\pi : \bar{M} \to M$ has degree $d$, then
$\area(\bar{M},\bar{h})=d\area(M,h$) and 
$\area(\bar{M},\bar{hyp})=d\area(M,hyp)$. So

$$\area(\bar{M},\bar{h}) \leq \frac{1}{2^{13}\pi}\area(\bar{M},\bar{hyp}).$$ 
Finally, since the universal cover of $(\bar{M},\bar{h})$ 
agrees with the universal cover of $(M,h)$, we can conclude by the first case.
\end{proof}

Now we prove Theorem \ref{maintheorem1}.

\begin{proof}[Proof of Theorem \ref{maintheorem1}]
Let $(M,hyp)$ be a closed hyperbolic  Riemannian surface of genus $g$.
Let $\delta$ be a small positive constant and $h$ another metric on~$M$ with 
$\area(M,h)\leq \delta \area(M,hyp).$
We will show that  if we take $\delta$ small enough (independently
from the metric $h$) then for any radius $R\geq 1$,
$$V_{(\Mtil,\htil)}(R)\geq  V_{\mathbb{H}^2}(R).$$

Indeed, let  $h'=\lambda^2h$ where $\lambda$ is a positive constant such that
$$\area(M,h')=\frac{1}{2^{13} \pi}\area (M,hyp).$$
By Theorem \ref{mainthm}, we have that for any radius $R\geq 0$,
$$V_{(\Mtil,\hprimetil)}(R)\geq \frac{1}{4\pi\ln2}V_{\mathbb{H}^2}(R\ln2).$$
Recall that 
$$\area(M,h')=\lambda^2\area(M,h) \leq \lambda^2\delta \area(M,hyp).$$
So 
$$\lambda^2\geq \frac{1}{2^{13}\pi\delta}.$$
On the other hand, we have 
$$V_{(\Mtil, \htil')} (\lambda R)= \lambda^2 V_{(\Mtil, \htil)} (R).$$
So 
$$V_{(\Mtil, \htil)} (R) \geq \frac{1}{4\pi\lambda^2\ln2}   V_{\mathbb{H}^2}(\lambda R \ln2).$$
\\
Now we choose $\lambda$ large  enough so  that for all $R\geq 1$ we have 
$$  \frac{1}{4\pi\lambda^2\ln2}  V_{\mathbb{H}^2}(\lambda R\ln2) \geq V_{\mathbb{H}^2}(R).$$
\\
\\
To see that such a $\lambda$ exists notice that for $R\geq 1$ we have
$$ \frac{1}{4\pi\lambda^2\ln2}  V_{\mathbb{H}^2}(\lambda R\ln2)\geq 
\frac{1}{8\pi\lambda^2\ln2}(e^{\frac{\lambda \ln2}{2}}e^{\frac{\lambda R\ln2}{2}}-2).$$
\\ \noindent
When $\lambda$ tends to infinity, the number $\frac{1}{8\pi\lambda^2\ln2}e^{\frac{\lambda \ln2}{2}}$ tends to infinity and so 
$$\frac{1}{8\pi\lambda^2\ln2}(e^{\frac{\lambda \ln2}{2}}e^{\frac{\lambda R\ln2}{2}}-2)\gg V_{\mathbb{H}^n}(R).$$
\noindent
Recall that to get $\lambda$ large enough it suffices to choose  $\delta$ small enough.
\\ 
Finally, we would like to point out that when R tends to zero we cannot find a $\lambda$ such that 
$$  \frac{1}{4\pi\lambda^2\ln2}  V_{\mathbb{H}^2}(\lambda R\ln2) \geq V_{\mathbb{H}^2}(R).$$
\end{proof}

\end{document}